\documentclass[12pt]{article}
\title{Dp-finite fields VI: the dp-finite Shelah conjecture}
\author{Will Johnson}

\usepackage{amsmath, amssymb, amsthm}    	
\usepackage{fullpage} 	
\usepackage{amscd}
\usepackage{hyperref}
\usepackage[all]{xy}
\usepackage{centernot}

\DeclareMathOperator*{\forkindep}{\raise0.2ex\hbox{\ooalign{\hidewidth$\vert$\hidewidth\cr\raise-0.9ex\hbox{$\smile$}}}}

\newcommand{\ACVF}{\operatorname{ACVF}}
\newcommand{\Jac}{\operatorname{Jac}}

\newcommand{\All}{\operatorname{All}}

\newcommand{\wt}{\operatorname{wt}}

\newcommand{\Frac}{\operatorname{Frac}}

\newcommand{\characteristic}{\operatorname{char}}
\newcommand{\Spec}{\operatorname{Spec}}
\newcommand{\MaxSpec}{\operatorname{MaxSpec}}

\newcommand{\Th}{\operatorname{Th}}

\newcommand{\res}{\operatorname{res}}
\newcommand{\Aut}{\operatorname{Aut}}

\newcommand{\tp}{\operatorname{tp}}

\newcommand{\val}{\operatorname{val}}

\newcommand{\dpr}{\operatorname{dp-rk}}

\newtheorem{theorem}{Theorem}[section] 
\newtheorem{lemma}[theorem]{Lemma}

\newtheorem{corollary}[theorem]{Corollary}
\newtheorem{fact}[theorem]{Fact}

\newtheorem{proposition}[theorem]{Proposition}
\newtheorem*{theorem-star}{Theorem}
\newtheorem*{lemma-star}{Theorem}
\newtheorem*{conjecture-star}{Conjecture}

\theoremstyle{definition}
\newtheorem{definition}[theorem]{Definition}

\theoremstyle{remark}
\newtheorem{remark}[theorem]{Remark}
\newtheorem{claim}[theorem]{Claim}

\newtheorem{warning}[theorem]{Warning}

\newtheorem*{acknowledgment}{Acknowledgments}

\newcommand{\Qq}{\mathbb{Q}}

\newcommand{\Rr}{\mathbb{R}}

\newcommand{\Zz}{\mathbb{Z}}

\newcommand{\Nn}{\mathbb{N}}

\newcommand{\Ff}{\mathbb{F}}

\newcommand{\Oo}{\mathcal{O}}
\newcommand{\mm}{\mathfrak{m}}
\newcommand{\pp}{\mathfrak{p}}
\newcommand{\qq}{\mathfrak{q}}

\newenvironment{claimproof}[1][\proofname]
               {
                 \proof[#1]
                 
               }
               {
                 \endproof
               }

\begin{document}
\maketitle

\begin{abstract}
  We prove the dp-finite case of the Shelah conjecture on NIP fields.
  If $K$ is a dp-finite field, then $K$ admits a non-trivial definable
  henselian valuation ring, unless $K$ is finite, real closed, or
  algebraically closed.  As a consequence, the conjectural
  classification of dp-finite fields holds.  Additionally, dp-finite
  valued fields are henselian.  Lastly, if $K$ is an unstable
  dp-finite expansion of a field, then $K$ admits a unique definable
  V-topology.
\end{abstract}

\section{Introduction}
A structure is \emph{dp-finite} if it has finite dp-rank.  We prove
the following facts about dp-finite fields and valued fields:
\begin{theorem}[dp-finite henselianity conjecture]
  Let $(K,v)$ be a dp-finite valued field.  Then $v$ is henselian.
\end{theorem}
\begin{theorem}[dp-finite Shelah conjecture]
  Let $K$ be a dp-finite field.  Then one of the following holds:
  \begin{itemize}
  \item $K$ is finite.
  \item $K$ is algebraically closed.
  \item $K$ is real closed.
  \item $K$ admits a definable non-trivial henselian valuation.
  \end{itemize}
\end{theorem}
By Theorems~3.3 and 3.11 of \cite{halevi-hasson-jahnke}, this implies
a classification of dp-finite fields up to elementary equivalence:
\begin{theorem}
  A field $K$ is dp-finite if and only if there is a henselian
  defectless valuation $v$ on $K$ such that
  \begin{itemize}
  \item The residue field $Kv$ is elementarily equivalent to
    $\Ff_p^{alg}$ or a local field of characteristic 0, i.e., a finite
    extension of $\Rr$ or $\Qq_p$.
  \item The value group $vK$ is dp-finite as an ordered abelian group.
  \item If $\characteristic(K) = p$, then the value group $vK$ is
    $p$-divisible.
  \item If $(K,v)$ has mixed characteristic $(0,p)$, then
    $[-v(p),v(p)] \subseteq p \cdot vK$.
  \end{itemize}
  Moreover, when the above conditions hold, the theory of $K$ is
  determined by the theory of $Kv$ and the theory of $vK$ (with $v(p)$
  named as a constant in the mixed characteristic case).
\end{theorem}
\begin{remark}
  ~
  \begin{enumerate}
  \item Dp-finite ordered abelian groups are classified in
    \cite{sd-groups-dolich-goodrick,sd-groups-farre,sd-groups-halevi-hasson}.
  \item Theorems~3.3 and 3.11 of \cite{halevi-hasson-jahnke} are
    phrased for strongly dependent fields, but the proof applies to
    dp-finite fields as well.  All strongly dependent fields are
    conjectured to be dp-finite (\cite{halevi-hasson-jahnke},
    Proposition~3.9).
  \item The valuation $v$ is not uniquely determined by $K$, and the
    map from $(\Th(Kv),\Th(vK))$ to $\Th(K)$ is many-to-one, rather
    than a bijection.  For example, if $(K,v_0) \models \ACVF$,
    we could take $v$ to be either $v_0$ or the trivial valuation.  In
    the first case, the value group $vK$ would be a divisible ordered
    abelian group; in the second case, it would be trivial.
  \end{enumerate}
\end{remark}

\begin{theorem}\label{thm:v-top-cor}
  Let $(K,+,\cdot,\ldots)$ be an unstable dp-finite field, possibly
  with extra structure.  Then $K$ admits a unique definable
  V-topology.
\end{theorem}
This implies the dp-finite Shelah and henselianity conjectures, by
(\cite{prdf4}, Proposition~6.4).

\subsection{Reduction to W-topologies}

In \cite{prdf,prdf2} we defined a ``canonical topology'' on any
unstable dp-finite field, and proved
\begin{fact}
  The canonical topology is a field topology.  If it is V-topological,
  then it is the unique definable V-topology.
\end{fact}
However, in \cite{prdf4}, we gave an example of an expanded field
$(K,+,\cdot,\ldots)$ of dp-rank 2, in which the canonical topology was
not a V-topology.  In \cite{prdf5} we introduced a class of ``finite
weight'' topological fields, or W-topological fields.  Field
topologies of weight 1 are exactly V-topologies.  We proved the
following:
\begin{fact}
  ~
  \begin{enumerate}
  \item The canonical topology on an unstable dp-finite field $K$ is a
    definable field topology of weight at most $\dpr(K)$.
  \item If $\tau$ is a field topology of finite weight (a W-topology),
    then there is at least one V-topological coarsening, i.e., a
    V-topology $\sigma$ that is coarser than $\tau$.
  \item If $\tau$ is a definable W-topology, then the V-topological
    coarsenings of $\tau$ are all definable.
  \item If $\tau$ is the canonical topology of an unstable dp-finite
    field $K$, then the V-topological coarsenings of $\tau$ are
    exactly the definable V-topologies on $K$.
  \end{enumerate}
\end{fact}
As a corollary, we obtained the existence part of
Theorem~\ref{thm:v-top-cor}.  Moreover, the uniqueness of the
definable V-topology was already known in characteristic $p > 0$, by
(\cite{prdf}, Lemma~2.6 and \cite{hhj-v-top}, Proposition~3.5).  So it
remains to prove uniqueness in characteristic 0.

In \cite{prdf2} (Proposition~5.17(4)), we used multiplicative
infinitesimals to prove:
\begin{fact}
  Let $K$ be an unstable dp-finite field, viewed as a topological
  field via the canonical topology.  For every neighborhood $U \ni
  1$, the set $U^2 = \{x^2 : x \in U\}$ is a neighborhood of 1.

  Equivalently, the squaring map $K^\times \to K^\times$ is an open map.
\end{fact}
Thus, everything reduces to the following statement purely about
W-topologies:
\begin{lemma}[= Corollary~\ref{sq-weird}]
  Let $(K,\tau)$ be a W-topological field of characteristic 0.  If the
  squaring map $K^\times \to K^\times$ is an open map, then $\tau$ has
  a unique V-topological coarsening.
\end{lemma}
This in turn comes from the following decomposition theorem:
\begin{theorem}[= Theorems~\ref{thm:non-loc1} and \ref{thm:non-loc2}] \label{thm-main}
  Let $(K,\tau)$ be a W-topological field.  Then there exist
  W-topological coarsenings $\tau_1, \ldots, \tau_n$ such that
  \begin{itemize}
  \item The $\tau_i$ are jointly independent and generate $\tau$.  In
    other words, the diagonal embedding
    \[ (K,\tau) \hookrightarrow (K,\tau_1) \times \cdots \times (K,\tau_n)\]
    is a homeomorphism onto its image, and the image is dense.
  \item Each $\tau_i$ has a unique V-topological coarsening, and this
    establishes a bijection between $\{\tau_1,\ldots,\tau_n\}$ and
    the set of V-topological coarsenings of $\tau$.
  \end{itemize}
\end{theorem}
We discuss the strategy for proving Theorem~\ref{thm-main} in
\S\ref{sec:where}, after introducing some machinery in
\S\ref{sec:machine}.

\subsection{Conventions}\label{sec:conventions}
This paper is a continuation of \cite{prdf5}, and we use its notions
of ``weight,'' and $W_n$-sets, -rings, and -topologies.  A
\emph{W-topology} is a topology of finite weight, i.e., a
$W_n$-topology for some $n$.

All rings will be commutative and unital.  If $R$ is a ring, then
$R^\times$ will denote the group of units of $R$, and $\Jac(R)$ will
denote the Jacobson radical.  If $R$ is an integral domain, then
$\Frac(R)$ will denote its field of fractions, and $\widetilde{R}$
will denote the integral closure of $R$ in
$\Frac(R)$.  We will tend to use the letter $\Oo$ for valuation rings,
and $\mm$ for maximal ideals of valuation rings.

In a list, $\ldots,\widehat{x},\ldots$ means ``omit $x$ from the
list.''  If $R$ is a ring and $x_1, \ldots, x_n$ are elements of an
$R$-module $M$, then the $x_i$ are \emph{$R$-independent} if no $x_i$
lies in the $R$-module generated by
$\{x_1,\ldots,\widehat{x_i},\ldots,x_n\}$.  Often $M = R$ or $M =
\Frac(R)$.  The ``cube-rank'' of (\cite{prdf5}, \S 2.1) is the maximum
length of an $R$-independent sequence in $M$.

If $K$ is a field and $A, B \subseteq K$, we will say that $A$ is
\emph{embeddable into} $B$ if there is $c \in K^\times$ such that $c
\cdot A \subseteq B$.  We say that $A$ and $B$ are
\emph{co-embeddable} if $A$ is embeddable into $B$ and $B$ is
embeddable into $A$.

All topologies will be Hausdorff non-discrete locally bounded ring
topologies on fields.  We think of a topology as the filter of
neighborhoods of 0, rather than the set of open sets.  If $\tau$ is a
topology, then $\tau^\perp$ will denote the ideal of bounded sets.  By
assumption, $\tau$ always intersects $\tau^\perp$.

We will say that $U$ ``defines'' or ``induces'' a topology $\tau$ if
$U$ is a locally bounded neighborhood in $\tau$, i.e., $U \in \tau
\cap \tau^\perp$.  In this case, $\{c U : c \in K^\times\}$ is a
filter basis for $\tau$, as well as an ideal basis for $\tau^\perp$
(Lemma~2.1(e) of \cite{PZ}).

If $R$ is a proper subring of $K$ and $\Frac(R) = K$, then $R$ induces
a topology $\tau_R$.  The non-zero ideals of $R$ form a neighborhood
basis, but we will predominantly use the basis $\{c R : c \in
K^\times\}$.

The terms ``local class,'' ``local sentence,'' and ``local
equivalence'' will be used as in \cite{PZ}.  In particular, a local
class is a class of topological fields defined by a set of local
sentences, and two topological fields are locally equivalent if they
satisfy the same local sentences.  Local sentences allow two types of
variables:
\begin{itemize}
\item Lower-case variables $a,b,c,\ldots,x,y,z$, which range over the
  field sort.
\item Upper-case variables $U, V, W, \ldots$, which range over $\tau$.
\end{itemize}
Quantification over $\tau$ is limited:
\begin{itemize}
\item Universal quantification $\forall U : \phi(U)$ is allowed only
  if $U$ occurs positively in $\phi(U)$.
\item Existential quantification $\exists U : \phi(U)$ is allowed only
  if $U$ occurs negatively in $\phi(U)$.
\end{itemize}
Because of this constraint, local sentences can be evaluated on a
filter basis for $\tau$: if $\tau_0$ is a filter basis for $\tau$, and
$\psi$ is a local sentence, then
\[ (K,\tau) \models \psi \iff (K,\tau_0) \models \psi.\]
  
Let $K$ be a field.  Consider the expansion of $(K,+,\cdot)$ by all
unary predicates.  Let $\All_K$ be the theory of the resulting object.
Henceforth, an ``ultrapower'' of $K$ will mean a monster model of
$\All_K$.  If $K^*$ is an ``ultrapower'' of $K$, and $U \subseteq K$,
then $U^*$ will denote the corresponding subset of $K^*$.  In the
structure $K^*$, the $K$-definable subsets of $K^*$ are exactly the
sets $U^*$.  A $\vee$-definable set will be a complement of a
type-definable set.  Type-definable and $\vee$-definable sets will
always be defined over small subsets of the ``ultrapower.''

\section{Topologies and $\vee$-definable rings} \label{sec:machine}

Fix a field $K$ and an ``ultrapower'' $K^*$.  We review the (easy)
dictionary between locally bounded ring topologies on $K$ and certain
$\vee$-definable subrings of $K^*$.
\begin{proposition}\label{p1}
  For every topology $\tau$ on $K$, there is a ring $R = R_\tau \subseteq K^*$ such
  that
  \begin{enumerate}
  \item\label{p1e1} $R$ is a filtered union $\bigcup_{B \in \tau^\perp} B^*$,
    i.e., the union of $B^*$ as $B$ ranges over $\tau$-bounded subsets
    of $K$.
  \item $R$ is $\vee$-definable over $K$.
  \item $R$ is a proper $K$-subalgebra of $K^*$.
  \item\label{p1e4} If $\tau$ and $\tau'$ are two topologies, then
    $\tau$ is coarser than $\tau'$ if and only if $R_{\tau} \supseteq
    R_{\tau'}$.
  \end{enumerate}
\end{proposition}
\begin{proof}
  Define $R$ as in point \ref{p1e1}.  Then $R$ is trivially
  $\vee$-definable over $K$.  For the other points, we use Lemma~2.1
  in \cite{PZ}.  If $B_1, B_2$ are two bounded sets, then $B_1 \cup
  B_2$ is bounded.  Therefore the union is filtered.  Any finite
  subset of $K$ is bounded.  Therefore $K \subseteq R$.  If $B_1, B_2$
  are bounded, then $B_1 - B_2$ and $B_1 \cdot B_2$ are bounded.
  Therefore $R$ is a $K$-subalgebra of $K^*$.  The set $K$ is not
  itself bounded.  By saturation of $K^*$, it follows that $R \ne
  K^*$.

  Lastly, for point \ref{p1e4} we prove the following chain of
  equivalent statements:
  \begin{enumerate}
  \item\label{he1} $\tau$ is coarser than $\tau'$, i.e., every
    $\tau$-open is a $\tau'$-open.
  \item\label{he2} $\tau \subseteq \tau'$, i.e., every
    $\tau$-neighborhood of 0 is a $\tau'$-neighborhood of 0.
  \item\label{he3} $(\tau')^\perp \subseteq \tau^\perp$, i.e., every
    $\tau'$-bounded set is $\tau$-bounded.
  \item\label{he4} $R_{\tau'} \subseteq R_\tau$.
  \end{enumerate}
  The equivalence (\ref{he1})$\iff$(\ref{he2}) is easy.  The
  equivalence (\ref{he2})$\iff$(\ref{he3}) holds because $\tau$ and
  $\tau^\perp$ determine each other:
  \begin{itemize}
  \item $B \in \tau^\perp$ if and only if $B$ is embeddable into every
    $U \in \tau$, in the sense that $\exists c \in K^\times : cB
    \subseteq U$.  This is Lemma~2.1(d) in \cite{PZ}.
  \item $U \in \tau$ if and only if every $B \in \tau^\perp$ is
    embeddable into $U$.  This holds by local boundedness of the
    topology.
  \end{itemize}
  Finally, the equivalence (\ref{he3})$\iff$(\ref{he4}) follows by
  compactness (and the fact that the unions are filtered).
\end{proof}
So $R_\tau$ is the ring of ``$K$-bounded'' elements in $K^*$.  One
could also define the group of ``$K$-infinitesimals'' as the
intersection $\bigcap_{U \in \tau} U^*$.  However, the ring of
$K$-bounded elements is more useful for our purposes.

\begin{proposition}\label{p2}
  For every topology $\tau$ on $K$, there is a topology $\tau^*$ on
  $K^*$ such that
  \begin{enumerate}
  \item\label{p2e1} $\tau^*$ is defined by the ring $R = R_\tau$ of
    Proposition~\ref{p1}.
  \item\label{p2e2} If $U$ is any set defining $\tau$, then $U^*$ defines
    $\tau^*$.
  \item\label{p2e3} If $\tau_1, \ldots, \tau_n$ are topologies on $K$, and
    $\tau_1^*, \ldots, \tau_n^*$ are the corresponding topologies on
    $K^*$, then there is a ``local equivalence'' in the sense of
    Prestel and Ziegler:
    \begin{equation*}
      (K^*,\tau_1^*,\ldots,\tau_n^*) \equiv (K,\tau_1,\ldots,\tau_n).
    \end{equation*}
  \end{enumerate}
\end{proposition}
\begin{proof}
  Fix any $U$ defining $\tau$.  Then $U$ is a bounded neighborhood of
  0, so $U^* \subseteq R_\tau$.  For every bounded set $B \in
  \tau^\perp$, there is $c \in K^\times$ such that $cB \subseteq U$.
  By saturation, there is $c \in (K^*)^\times$ such that $cB^*
  \subseteq U^*$ for all $B \in \tau^\perp$.  Thus $cR_\tau \subseteq
  U^*$.  So $R_\tau$ and $U$ are co-embeddable.

  By the proof of Lemma~2.3 in \cite{PZ}, $U^*$ defines a topology
  $\tau^*$ on $K^*$, with $(K^*,\tau^*) \equiv (K,\tau)$.  By
  co-embeddability of $R_\tau$ and $U$, the topology $\tau^*$ is also
  defined by $R_\tau$.  In particular, $\tau^*$ is independent of the
  choice of $U$.

  Now suppose $\tau_1, \ldots, \tau_n$ are fixed topologies on $K$.
  For each $i$, choose $U_i \subseteq K$ defining $\tau_i$.  Then
  there is an elementary equivalence
  \[ (K^*,+,\cdot,U_1^*,\ldots,U_n^*) \equiv (K,+,\cdot,U_1,\ldots,U_n).\]
  By Corollary~2.4 in \cite{PZ}, this implies the desired local
  equivalence, as each $U_i^*$ defines $\tau_i^*$.
\end{proof}

\begin{proposition}\label{wdict-1}
  Let $\tau$ be a topology on $K$, and let $R_\tau$ be the
  corresponding $\vee$-definable ring.
  \begin{enumerate}
  \item $\tau$ is a $W_n$-topology if and only if $R_\tau$ is a
    $W_n$-ring.
  \item In particular, the weight of $\tau$ equals the weight of
    $R_\tau$.  (The weight may be infinite.)
  \end{enumerate}
\end{proposition}
\begin{proof}
  First suppose $\tau$ is a $W_n$-topology.  By definition
  (Definition~3.3 in \cite{prdf5}), there is a bounded $W_n$-set $U
  \subseteq K$.  Then $U^*$ is a $W_n$-set in $K^*$, and $U^*
  \subseteq R_\tau$, implying that $R_\tau$ is a $W_n$-set and a
  $W_n$-ring.

  Conversely, if $R_\tau$ is a $W_n$-ring, then it defines a
  $W_n$-topology by Proposition~3.6 in \cite{prdf5}.  Therefore,
  $\tau^*$ is a $W_n$-topology.  By the local equivalence $(K,\tau)
  \equiv (K^*,\tau^*)$, it follows that $\tau$ is a $W_n$-topology.

  This proves the first point.  Then we know that $\wt(\tau) \le n
  \iff \wt(R_\tau) \le n$ for all $n \in \Nn$, implying that
  $\wt(\tau) = \wt(R_\tau)$.
\end{proof}

\begin{lemma}\label{duh}
  Let $R_1, R_2$ be two $K$-subalgebras of $K^*$, both
  $\vee$-definable over $K$.  If $R_1$ and $R_2$ are co-embeddable,
  then $R_1 = R_2$.  More generally, if $R_1$ is embeddable into
  $R_2$, then $R_1 \subseteq R_2$.
\end{lemma}
\begin{proof}
  Each $R_i$ is a filtered union of its $K$-definable sets.  Suppose
  $cR_1 \subseteq R_2$ for some $c \in (K^*)^\times$.  Then for every
  $K$-definable subset $U \subseteq R_1$, we have $cU \subseteq R_2$.
  By saturation, there must be some $K$-definable subset $V \subseteq
  R_2$ such that $cU \subseteq V$.  The scalar $c$ is from
  $(K^*)^\times$, but $U$ and $V$ are $K$-definable, and $K \preceq
  K^*$.  Therefore we can find $e \in K^\times$ such that $eU
  \subseteq V$.  Then $U \subseteq e^{-1}V \subseteq e^{-1}R_2
  \subseteq R_2$, because $e^{-1} \in K \subseteq R_2$.  As $U$ was
  arbitrary, $R_1 \subseteq R_2$.
\end{proof}

\begin{proposition}\label{wdict-2}
  Let $R$ be a subring of $K^*$ that is $\vee$-definable over $K$, and
  satisfies $K \subseteq R \subsetneq \Frac(R) = K^*$.
  \begin{enumerate}
  \item $R$ is of the form $R_\tau$ for some topology on $K$ if and
    only if $R$ is co-embeddable with a definable set.
  \item $R$ is of the form $R_\tau$ for some $W_n$-topology on $K$ if
    and only if $R$ is a $W_n$-ring.
  \item $R$ is of the form $R_\tau$ for some V-topology on $K$ if and
    only if $R$ is a valuation ring.
  \end{enumerate}
\end{proposition}
\begin{proof}
  \begin{enumerate}
  \item If $R = R_\tau$, then $R$ is co-embeddable with $U^*$ for any
    $U$ defining $\tau$, by
    Proposition~\ref{p2}(\ref{p2e1}-\ref{p2e2}).

    Conversely, suppose $R$ is co-embeddable with a definable set $D$.
    Write $D$ as $\phi(K^*;\vec{b})$ for some formula
    $\phi(x;\vec{y})$ (in the language of $\All_K$), and some
    parameters $\vec{b}$ from $K^*$.  Let $S$ be the set of $\vec{c}$
    in $K^*$ satisfying the equivalent conditions
    \begin{itemize}
    \item $\phi(K^*;\vec{c})$ is co-embeddable with $D$
    \item $\phi(K^*;\vec{c})$ is co-embeddable with $R$.
    \end{itemize}
    From the first characterization, $S$ is definable.  From the
    second characterization, $S$ is $\Aut(K^*/K)$-invariant.
    Therefore $S$ is $K$-definable, and we can find $\vec{c} \in
    S(K)$.  Let $U = \phi(K;\vec{c})$.  Then $U^* =
    \phi(K^*;\vec{c})$, and $U^*$ is co-embeddable with $R$.

    Let $\tau_R$ denote the topology on $K^*$ induced by $R$, as in
    Example~1.2 of \cite{PZ}.  Then $U^*$ defines $\tau_R$, by
    co-embeddability.  The fact that $U^*$ defines a locally bounded
    ring-topology on $K^*$ is expressed by a first-order sentence in
    the structure $(K^*,+,\cdot,U^*)$.  The same sentence holds in
    $(K,+,\cdot,U)$, and so $U$ defines a topology $\tau$ on $K$.  Let
    $R' = R_\tau$ be the corresponding $\vee$-definable subring on
    $K^*$.  Then $R'$ is co-embeddable with $U^*$, by
    Proposition~\ref{p2}.  Therefore $R'$ is co-embeddable with $R$.
    By Lemma~\ref{duh}, $R' = R_\tau$.
  \item If $\tau$ is a $W_n$-topology, then $R_\tau$ is a $W_n$-ring
    by Proposition~\ref{wdict-1}.  Conversely, suppose $R$ is a
    $W_n$-ring.  By assumption, $R$ is $\vee$-definable.  By
    Proposition~4.1 in \cite{prdf5}, $R$ is co-embeddable with a
    definable set.  By the previous point, $R = R_\tau$ for some
    topology $\tau$.  By Proposition~\ref{wdict-1}, the fact that $R$
    is a $W_n$-ring forces $\tau$ to be a $W_n$-topology.
  \item This is the $n=1$ case of the previous point. \qedhere
  \end{enumerate}
\end{proof}

\begin{lemma}\label{esimals}
  If $\tau$ is a V-topology on $K$ and $\Oo = R_\tau$ is the
  corresonding valuation ring, then the maximal ideal $\mm \subseteq
  \Oo$ is the set of $K$-infinitesimals, i.e., $\mm = \bigcap_{U \in
    \tau} U^*$.
\end{lemma}
\begin{proof}
  Fix a bounded neighborhood $B \in \tau \cap \tau^\perp$.  Let $C =
  \{x \in K : 1/x \notin B\}$.  Then $C^{-1} \cap B = \emptyset$ and
  $C \cup B^{-1} = K$, so that $C$ is a bounded neighborhood of 0.
  (This follows by properties of V-topologies, such as the
  \emph{definition} of V-topologies given in \S 3 of \cite{PZ}.)  Then
  \begin{align*}
    \Oo &= \bigcup_{U \in \tau^\perp} U^* = \bigcup_{a \in K^\times} aB^*. \\
    \bigcap_{U \in \tau} U^* &= \bigcap_{a \in K^\times} aC^* = \bigcap_{a \in K^\times} a^{-1}C^*.
  \end{align*}
  Therefore, the following are equivalent for $x \in K^*$:
  \begin{align*}
    x \in \mm &\iff 1/x \notin \Oo \iff \left(\forall a \in K^\times : 1/x \notin aB^*\right) \\
    & \iff \left(\forall a \in K^\times : 1/(ax) \notin B^* \right) \iff \left( \forall a \in K^\times : ax \in C^*\right) \\
    & \iff \left( \forall a \in K^\times : x \in a^{-1}C^*\right) \iff x \in \bigcap_{a \in K^\times} a^{-1}C^* \iff x \in \bigcap_{U \in \tau} U^*. \qedhere
  \end{align*}
\end{proof}

\begin{lemma}\label{almost}
  Let $X \subseteq K^*$ be $\vee$-definable over a small set.  If $X$
  has finite orbit under $\Aut(K^*/K)$, then $X$ is $\vee$-definable
  over $K$.
\end{lemma}
This is somewhat well-known, but we include the proof for
completeness.
\begin{proof}
  We first claim that $X$ is $\Aut(K^*/K)$-invariant ($K$-invariant).
  Recall that if $\approx$ is a $K$-invariant equivalence relation
  with boundedly-many equivalence classes, and $a \equiv_K b$, then $a
  \approx b$, because $K$ is a model.\footnote{This is a well-known
    fact about Lascar strong type, and is easy to prove by using a
    global coheir of $\tp(a/K)$ to build a sequence $c_1, c_2, \ldots$
    such that both $a,c_1,c_2,\ldots$ and $b,c_1,c_2,\ldots$ are
    $K$-indiscernible.}  Let $a \approx b$ indicate that $a \in
  \sigma(X) \iff b \in \sigma(X)$ for all $\sigma \in \Aut(K^*/K)$.
  Then $\approx$ is $K$-invariant, with finitely many equivalence
  classes.  If $X$ itself fails to be $K$-invariant, then there are
  $a$ and $b$ such that $a \equiv_K b$ but $a \in X$ and $b \notin X$.
  Then $a \not \approx b$, a contradiction.

  Thus, $X$ is $K$-invariant.  Now, any $\vee$-definable $K$-invariant
  set is $\vee$-definable over $K$.\footnote{An equivalent,
    better-known statement is that a type-definable $K$-invariant set
    is type-definable over $K$.}  Indeed, if $X$ is $\vee$-definable
  over a small parameter set $B \supseteq K$, then $X$ corresponds to
  some open set $U$ in the space $S_n(B)$ of $n$-types over $B$.  The
  $K$-invariance means that $U$ is the preimage of some set $U'
  \subseteq S_n(K)$ under the continuous surjection $S_n(B)
  \twoheadrightarrow S_n(K)$.  The complement $S_n(K) \setminus U'$ is
  the image of the closed set $S_n(B) \setminus U$, so $S_n(K) \setminus U'$ is
  closed, $U'$ is open, and $X$ is $\vee$-definable over $K$.
\end{proof}

\begin{proposition}\label{int-loc}
  Let $\tau$ be a W-topology on $K$, and let $R = R_\tau$ be the
  corresponding subring of $K^*$.
  \begin{itemize}
  \item There is a unique W-topology $\widetilde{\tau}$ such that the
    corresponding ring $R_{\widetilde{\tau}}$ is the integral closure
    of $R$.
  \item There are W-topologies $\tau_1, \ldots, \tau_n$ such that the
    corresponding rings $R_{\tau_1}, \ldots, R_{\tau_n}$ are exactly
    the localizations of $R$ at its maximal ideals.
  \end{itemize}
\end{proposition}
\begin{proof}
  By Propositions~4.7 and 4.8 in \cite{prdf5}, the integral closure
  and the localizations are $\vee$-definable.  The integral closure is
  clearly $K$-invariant, and the localizations can at most be permuted
  by $\Aut(K^*/K)$.  By Lemma~\ref{almost}, the localizations and the
  integral closure are $\vee$-definable over $K$.  They are larger
  than $R$, so they contain $K$ and are $K$-algebras.  By Lemma~2.7 in
  \cite{prdf5}, they are rings of finite weight.  By
  Proposition~\ref{wdict-2}, they come from W-topologies on $K$.
\end{proof}

\begin{definition}\label{def:comps}
  Let $\tau$ be a W-topology on $K$.
  \begin{itemize}
  \item The \emph{integral closure} of $\tau$ is the topology
    $\widetilde{\tau}$ of Proposition~\ref{int-loc}.
  \item The \emph{local components} of $\tau$ are the topologies
    $\tau_1, \ldots, \tau_n$ of Proposition~\ref{int-loc}.
  \end{itemize}
\end{definition}

Recall from (\cite{prdf5}, Proposition~2.12) that if $R$ is a
$W_n$-ring, then the integral closure $\widetilde{R}$ is a
multi-valuation ring, a finite intersection of valuation rings.
\[ \Oo_1 \cap \cdots \cap \Oo_n.\]
If the $\Oo_i$ are chosen to be pairwise incomparable, then the
$\Oo_i$ are exactly the localizations of $\widetilde{R}$ at its
maximal ideals (\cite{prdf2}, Corollary~6.7).

\begin{proposition}\label{v-coarses}
  Let $\tau$ be a W-topology on $K$.  Then the V-topological
  coarsenings of $\tau$ are exactly the local components of the
  integral closure.

  In other words, if $R \subseteq K^*$ is the corresponding
  $\vee$-definable ring, and we write $\widetilde{R}$ as an
  intersection of pairwise incomparable valuation rings
  \[ \widetilde{R} = \Oo_1 \cap \cdots \cap \Oo_n,\]
  then the $\Oo_i$ are exactly the $\vee$-definable rings
  corresponding to the V-topological coarsenings of $\tau$.
\end{proposition}
\begin{proof}
  By Proposition~\ref{wdict-2} and Proposition~\ref{p1}(\ref{p1e4}),
  the V-topological coarsenings of $\tau$ correspond exactly to the
  valuation rings $\Oo$ on $K$ with the following properties:
  \begin{itemize}
  \item $\Oo$ is non-trivial, i.e., $\Oo \ne K^*$.
  \item $\Oo$ is $\vee$-definable over $K$.
  \item $\Oo$ contains $R$.
  \end{itemize}
  As in the proof of Proposition~\ref{int-loc}, the $\Oo_i$ certainly
  have these properties.  Let $\Oo$ be some other valuation ring with
  these properties.  Then $\Oo \supseteq \widetilde{R} = \Oo_1 \cap
  \cdots \cap \Oo_n$.  By Corollary~6.8 in \cite{prdf2}, there is some
  $i$ such that $\Oo \supseteq \Oo_i$.  Then $\Oo$ is a coarsening of
  $\Oo_i$.  By non-triviality, $\Oo$ and $\Oo_i$ induce the same
  topology, so they are co-embeddable.  By Lemma~\ref{duh}, $\Oo =
  \Oo_i$.
\end{proof}

\subsection{Local W-topologies}

\begin{definition}
  A $W_n$-topology on $K$ is \emph{local} if for every bounded set $B
  \subseteq K$, there is a bounded set $C \subseteq K$ such that
  \[ \forall x \in B : (1/x \in C \text{ or } 1/(1-x) \in C).\]
\end{definition}
\begin{proposition} \label{local-dict}
  Let $K^*$ be an ``ultrapower'' of $K$.  Let $\tau$ be a W-topology
  on $K$, and let $R$ be the associated ring in $K^*$.  Then $\tau$ is
  local if and only if $R$ is a local ring.
\end{proposition}
\begin{proof}
  Let $\gamma(X,Y)$ stand for
  \[ \forall x \in X : 1/x \in Y \text{ or } 1/(1-x) \in Y.\]
  Then the following statements are equivalent:
  \begin{enumerate}
  \item \label{dirt1} $R$ is a local ring.
  \item \label{dirt2} For every $x \in R$, at least one of $x$ or $1 -
    x$ is in $R^\times$.
  \item \label{dirt3} $\gamma(R,R)$.
  \item \label{dirt4} For every bounded $B \subseteq K$, we have
    $\gamma(B^*,R)$.
  \item \label{dirt5} For every bounded $B \subseteq K$, there is
    bounded $C \subseteq K$ such that $\gamma(B^*,C^*)$.
  \item \label{dirt6} For every bounded $B \subseteq K$, there is
    bounded $C \subseteq K$ such that $\gamma(B,C)$.
  \item \label{dirt7} $\tau$ is a local W-topology.
  \end{enumerate}
  The equivalences are proven as follows:
  \begin{itemize}
  \item (\ref{dirt1})$\iff$(\ref{dirt2}): well-known commutative
    algebra.
  \item (\ref{dirt2})$\iff$(\ref{dirt3}): the definition of
    $\gamma(-,-)$.
  \item (\ref{dirt3})$\iff$(\ref{dirt4}): $R$ is covered by the $B^*$.
  \item (\ref{dirt4})$\iff$(\ref{dirt5}): $R$ is a filtered union of
    the $C^*$, and the structure is saturated.
  \item (\ref{dirt5})$\iff$(\ref{dirt6}): $K \preceq K^*$.
  \item (\ref{dirt6})$\iff$(\ref{dirt7}): the definition of ``local
    W-topology.''  \qedhere
  \end{itemize}
\end{proof}
For example, if $\tau$ is a W-topology, then the local components of
$\tau$ (Definition~\ref{def:comps}) are local W-topologies.
\begin{warning}\label{ww}
  \begin{enumerate}
  \item \label{ww1} A local $W_n$-ring need not induce a local
    $W_n$-topology.  For example, let $\Oo_1$ and $\Oo_2$ be the two
    valuation rings $\Zz[i]_{(2+i)}$ and $\Zz[i]_{(2-i)}$ on the field
    $K = \Qq(i)$.  The residue fields are both isomorphic to
    $\Zz/(5)$.  Let $R$ be the set of $x \in \Oo_1 \cap \Oo_2$ such
    that $\res_1(x) = \res_2(x)$.  Then $R$ is a $W_2$-ring, because
    it contains the valuation ring $\Zz_{(5)}$ on the subfield $\Qq$,
    and $[\Qq(i) : \Qq] = 2$.  (See Lemma~2.7 in \cite{prdf5}.)
    Additionally, $R$ is a local ring---the maximal ideal is the set
    of $x$ such that $\res(x) = 0$.  However, $R$ is co-embeddable
    with $\Oo_1 \cap \Oo_2$, and induces a $V^2$-topology, which is
    not local.
  \item \label{ww2} Conversely, a non-local $W_n$-ring can induce a
    local $W_n$-topology.  For example, if $\Oo_1$ and $\Oo_2$ are two
    incomparable but dependent valuation rings, then $\Oo_1 \cap
    \Oo_2$ is a non-local $W_2$-ring which induces the same V-topology
    as $\Oo_1 \cdot \Oo_2$.  V-topologies are local.
  \end{enumerate}
\end{warning}

We shall need the fact that local $W_n$-topologies form a local class,
in the sense of \cite{PZ}.  The following lemma helps translate
statements about bounded sets into statements about neighborhoods.
\begin{lemma}\label{bound-formulas}
  Let $(K,\tau)$ be a field with a locally bounded ring topology.  As
  usual $\tau$ denotes the filter of neighborhoods of 0 and
  $\tau^\perp$ denotes the ideal of bounded sets.  Let $\phi(X)$ be a
  formula in $X \subseteq K$, depending positively on $X$.  Then
  \[ \exists B \in \tau^\perp : \phi(B)\]
  is equivalent to
  \begin{equation}
    \forall U \in \tau ~ \exists c \in K^\times : \phi(cU). \label{uc}
  \end{equation}
  Dually, if $\phi(X)$ is a formula in which $X$ appears negatively,
  then
  \[ \forall B \in \tau^\perp : \phi(B)\]
  is equivalent to
  \[ \exists U \in \tau ~ \forall c \in K^\times : \phi(cU).\]
\end{lemma}
\begin{proof}
  We prove the first claim; the other follows formally.  Suppose there
  is a bounded set $B$ such that $\phi(B)$ holds.  By definition of
  ``bounded,'' for any $U \in \tau$ there is $e \in K^\times$ such
  that $eB \subseteq U$.  Then $e^{-1} U \supseteq B$, and so
  $\phi(e^{-1}U)$ is true.  This proves (\ref{uc}).

  Conversely, suppose (\ref{uc}) holds.  By local boundedness, there
  is $U \in \tau$ such that $U$ is bounded.  By (\ref{uc}), there is
  $c \in K^\times$ such that $\phi(cU)$ holds.  Then $B = cU$ is
  bounded, and $\phi(B)$ holds.
\end{proof}

\begin{proposition}\label{oh-local}
  The class of local $W_n$-topologies is cut out by a local sentence,
  and is therefore a local class.
\end{proposition}
\begin{proof}
  The class of $W_n$-topologies is defined by a local sentence
  (\cite{prdf5}, Remark~3.4).  As in the proof of
  Proposition~\ref{local-dict}, let $\gamma(X,Y)$ be the formula
  \[ \forall x \in X : (1/x \in Y \text{ or } 1/(1-x) \in Y).\]
  This formula is negative in $X$ and positive in $Y$.  Then
  $(K,\tau)$ is local if and only if
  \[ \forall B \in \tau^\perp ~ \exists C \in \tau^\perp : \gamma(B,C).\]
  By two applications of Lemma~\ref{bound-formulas}, this is
  equivalent to
  \begin{equation*}
    \exists U \in \tau ~ \forall c \in K^\times ~ \forall V \in \tau ~
    \exists e \in K^\times : \gamma(cU,eV).
  \end{equation*}
  This is a local sentence.  In detail, it is
  \begin{align*}
    & \exists U \in \tau ~ \forall c \in K^\times ~ \forall V \in \tau ~
    \exists e \in K^\times ~ \forall x : \\ & (x/c \in U) \rightarrow
    (1/(xe) \in V \text{ or } 1/((1-x)e) \in V). \qedhere
  \end{align*}
\end{proof}

\begin{remark}
  Suppose one is only interested in locally bounded ring topologies.
  Because of Lemma~\ref{bound-formulas}, one could extend Prestel and
  Ziegler's notion of ``local sentence'' to allow quantification over
  bounded sets, i.e., over $\tau^\perp$, subject to the constraints:
  \begin{itemize}
  \item Universal quantification $\forall B \in \tau^\perp : \phi(B)$ is allowed only
    if $B$ occurs negatively in $\phi(B)$.
  \item Existential quantification $\exists B \in \tau^\perp : \phi(B)$ is allowed only
    if $B$ occurs positively in $\phi(B)$.
  \end{itemize}
  These are opposite to the constraints on quantification over
  neighborhoods.
\end{remark}

\subsection{Where we are going} \label{sec:where}

We give an outline of the remainder of the paper, in reverse order,
working backwards from our goal.

Let $\tau$ be a W-topology on $K$.  Let $\tau_1, \ldots, \tau_n$ be
the local components of $\tau$.  For reasons discussed in the
introduction, we would like to prove the following two statements:
\begin{enumerate}
\item \label{st1} The $\tau_i$ are jointly independent, and generate $\tau$.
\item \label{st2} Each $\tau_i$ has a unique V-topological
  coarsening, and this establishes a bijection between
  $\{\tau_1,\ldots,\tau_n\}$ and the set of V-topological coarsenings.
\end{enumerate}
One can reduce (\ref{st1}) to (\ref{st2}) using the independence
criterion of (\cite{prdf5}, Theorem~7.16), more or less.

As for (\ref{st2}), it translates into a statement about rings.  Fix
an ``ultrapower'' $K^*$, and let $R$ be the ring corresponding to
$\tau$.  Let $\widetilde{R}$ be its integral closure.  Then
(\ref{st2}) translates into
\begin{enumerate}
  \setcounter{enumi}{2}
\item For any maximal ideal $\pp$ of $R$, there is a unique maximal
  ideal $\mm$ of $\widetilde{R}$ such that $\widetilde{R_\pp} =
  \widetilde{R}_\mm$.  This establishes a bijection between the
  maximal ideals of $R$ and the maximal ideals of $\widetilde{R}$.
\end{enumerate}
By commutative algebra, we reduce to the case where $R$ and $\tau$ are local:
\begin{enumerate}
  \setcounter{enumi}{3}
\item If $\tau$ is local, then $\tau$ has a unique V-topological coarsening.
\item \label{zloc} If $R$ is local, then $\widetilde{R}$ is a valuation ring.
\end{enumerate}
This cannot be proven purely by commutative algebra---the $W_2$-ring
of Warning~\ref{ww}(\ref{ww1}) is local, but its integral closure is
not.  We must use the fact that $R$ is a $\vee$-definable ring induced
by a topology.

Assume henceforth that $R$ and $\tau$ are local.  By pushing the
commutative algebra a little further, (\ref{zloc}) reduces to\ldots
\begin{enumerate}
  \setcounter{enumi}{5}
\item If $\Oo_1, \ldots, \Oo_n$ are the incomparable valuation rings
  whose intersection is $\widetilde{R}$, and $\mm_i$ is the maximal
  ideal of $\Oo_i$, then
  \[ \mm_i \cap R \centernot\subseteq \mm_j \cap R\]
  for $i \ne j$.
\end{enumerate}
Translated back into topology, this says the following
\begin{enumerate}
  \setcounter{enumi}{6}
\item If $\tau_1$ and $\tau_2$ are two distinct V-topological
  coarsenings of $\tau$, then the following \emph{can't} happen:
  $\tau_1$ induces a finer topology than $\tau_2$ on each
  $\tau$-bounded set.
\end{enumerate}
This makes intuitive sense, as $\tau_1$ and $\tau_2$ are independent
V-topologies.  To make this intuition precise, we need $\tau$-bounded
sets to be big enough:
\begin{enumerate}
  \setcounter{enumi}{7}
\item In the topology on $K^*$ induced by $\widetilde{R}$, the closure
  of $R$ includes the Jacboson radical $\Jac(\widetilde{R}) = \mm_1
  \cap \cdots \cap \mm_n$.
\end{enumerate}
In the case of DV-topologies, there was a trick to prove that $R$ is
dense in its integral closure $\Oo$; see \S 5.4 of \cite{prdf4}.  A
variant of that method works in our setting, but we need the following
configuration:
\begin{enumerate}
  \setcounter{enumi}{8}
\item \label{stepX} For any non-zero $a$ there are $y_1, \ldots, y_n$
  such that
  \begin{itemize}
  \item For each $i$, $ay_i$ is not in the $R$-module generated by
    $\{y_1,\ldots,\widehat{y_i},\ldots,y_n\}$.
  \item For fixed $i$, the elements $y_1, \ldots, y_n$ all have the
    same valuation with respect to $\Oo_i$.
  \end{itemize}
\end{enumerate}
Ignoring the factor of $a$, we need a sequence $y_1, \ldots, y_n$
which is $R$-independent in the sense of \S\ref{sec:conventions}, but
which is very far from being $\widetilde{R}$-independent---in fact
$y_i \in \widetilde{R} y_j$ for all $i, j$.  Such a sequence can be
obtained by scrambling an $R$-independent sequence by a random matrix
from $GL_n(\Qq)$:
\begin{enumerate}
  \setcounter{enumi}{9}
\item Let $R$ be a $W_n$-ring, extending $\Qq$ for simplicity.  Let $x_1, \ldots, x_n$
  be an $R$-independent sequence in $R$, and let $\vec{y} = \mu \cdot \vec{x}$,
  where $\mu$ is a ``random'' matrix from $GL_n(\Qq)$.
  \begin{itemize}
  \item If $R$ is local, then $\vec{y}$ is an $R$-independent
    sequence.
  \item If $R$ is a multi-valuation ring, then $y_i \in R y_j$ for all
    $i, j$.
  \end{itemize}
\end{enumerate}
Thus, local $W_n$-rings and multi-valuation rings behave differently,
and this difference can be leveraged to prove (\ref{stepX}).

\section{Scrambling and density} \label{sec:tech}
For the entirety of \S\ref{sec:tech}, we will work over a fixed
infinite field $K_0$.  All fields will extend $K_0$, all rings will be
$K_0$-algebras, and all valuations will be trivial on $K_0$.

\subsection{Scrambling} \label{sec:scramble}
\begin{lemma} \label{basic-scramble}
  Let $(K,\val_1,\ldots,\val_m)$ be a multi-valued field.
  For any $z, w \in K$, there is $c \in K_0$ such that for
  all $i$,
  \begin{equation}
    \val_i(z - cw) = \min(\val_i(z),\val_i(w)). \label{tri}
  \end{equation}
\end{lemma}
\begin{proof}
  For each $i$, let $\res_i : K \to k_i \cup \{\infty\}$ be the
  residue map of $\val_i$, extended by setting $\res_i(x) = \infty$
  when $\val_i(x) < 0$.

  If $w = 0$, then any $c$ works.  Otherwise, take $c \in
  K_0$ such that $c \ne \res_i(z/w)$ for all $i$.  This is possible
  since $K_0$ is infinite.  If (\ref{tri}) fails, then by the strong
  triangle inequality
  \[ \val_i(z - cw) > \val_i(z) = \val_i(w).\]
  Then $\val_i(z/w - c) = \val_i(z - cw) - \val_i(w) > 0$.  This
  implies $\res_i(z/w) = \res_i(c) = c$, contradicting the choice of
  $c$.
\end{proof}
Let $(K,\val_1,\ldots,\val_m)$ be a multi-valued field.  Let $\vec{x}$
be an $n$-tuple in $K^n$.  Say that $\vec{x}$ is \emph{scrambled} if
\begin{align*}
  \val_1(x_1) & = \val_1(x_2) = \cdots = \val_1(x_n) \\
  \val_2(x_1) & = \val_2(x_2) = \cdots = \val_2(x_n) \\
  \cdots \\
  \val_m(x_1) & = \val_m(x_2) = \cdots = \val_m(x_n).
\end{align*}
If $\mu \in GL_n(K)$, say that $\mu$ \emph{scrambles} $\vec{x}$ if
$\mu \cdot \vec{x}$ is scrambled.
\begin{lemma}\label{scramble-weak}
  Let $(K,\val_1,\ldots,\val_m)$ be a multi-valued field.  Then any
  $n$-tuple $\vec{x} \in K^n$ can be scrambled by an element of
  $GL_n(K_0)$.
\end{lemma}
\begin{proof}
  This follows by repeated applications of Lemma~\ref{basic-scramble}.
  In more detail, define the ``discrepancy'' of $\vec{x}$ to be the size
  of the set
  \[ \{(i,j) ~|~ \exists k : \val_i(x_j) > \val_i(x_k)\}.\]
  We may assume that $\vec{x}$ has minimal discrepancy in the coset
  $GL_n(K_0) \cdot \vec{x}$.  If $\vec{x}$ has discrepancy 0, then it
  is scrambled and we are done.  Otherwise, we may assume
  \[ \val_1(x_1) > \val_1(x_2).\]
  By Lemma~\ref{basic-scramble} there is $c \in K_0$ such that for all $i$,
  \[ \val_i(x_1 - cx_2) = \min(\val_i(x_1),\val_i(x_2)).\]
  Let $\vec{y} = (x_1 - cx_2, x_2, x_3, \ldots, x_n)$.  Then $\vec{y}
  \in GL_n(K_0) \cdot \vec{x}$, and $\vec{y}$ has lower discrepancy, a
  contradiction.
\end{proof}
\begin{lemma}\label{scramble-strong}
  For every $n, m$, there is a finite set $G_{n,m} \subseteq
  GL_n(K_0)$ with the following property.  Let $K$ be a field, and let
  $\val_1, \ldots, \val_m$ be valuations on $K$.  Then every $n$-tuple
  $\vec{x} \in K^n$ is scrambled by an element of $G_{n,m}$.
\end{lemma}
\begin{proof}
  Lemma~\ref{scramble-weak} and compactness.
\end{proof}
\begin{lemma}\label{spinaround}
  Let $R$ be a local integral domain.  If $x_1, \ldots, x_n \in R$ is
  an $R$-independent sequence in $\Frac(R)$, and $\vec{y} \in GL_n(K_0) \cdot
  \vec{x}$, then $y_1, \ldots, y_n$ is also $R$-independent.
\end{lemma}
\begin{proof}
  Let $\mm$ be the maximal ideal of $R$, and $k$ be the residue field
  $R/\mm$.  Let $M$ be the $R$-submodule of $\Frac(R)$ generated by
  the $x_i$.  Let $\xi_i$ be the image of $x_i$ in the $k$-vector
  space $M/\mm M$.  The $\xi_i$ generate $M/\mm M$, and are
  $R$-independent by Nakayama's lemma.  Therefore $M / \mm M$ has
  dimension $n$ over $k$.  Write $\vec{y}$ as $\mu \cdot \vec{x}$ for
  some $\mu \in GL_n(K_0)$.  Then $y_1, \ldots, y_n$ also generate
  $M$.  Let $\{z_1,\ldots,z_m\}$ be a minimal subset of
  $\{y_1,\ldots,y_n\}$ generating $M$.  Then $\vec{z}$ is
  $R$-independent.  The Nakayama's lemma argument shows $m = \dim_k
  M/\mm M = n$.  Therefore $\{y_1,\ldots,y_n\} = \{z_1,\ldots,z_m\}$,
  and the $\vec{y}$ are $R$-independent.
\end{proof}

\begin{lemma}\label{re-slide}
  Let $R$ be an integral domain of finite weight.  Suppose $R \ne
  \Frac(R) =: K$, and the induced topology on $R$ is local of weight
  $n$.  Let $\widetilde{R}$ be the integral closure of $R$.  Write
  $\widetilde{R}$ as $\Oo_1 \cap \cdots \cap \Oo_m$, where the $\Oo_i$ are
  pairwise incomparable valuation rings.  Let $\val_i$ be the
  valuation associated to $i$.  Then for any non-zero $a \in R$, we
  can find $y_1, \ldots, y_n \in K$ such that
  \begin{itemize}
  \item The tuple $\vec{y}$ is scrambled with respect to
    $(K,\val_1,\ldots,\val_n)$, i.e., $\val_i(y_j) = \val_i(y_k)$ for
    all $i, j, k$.
  \item For any $i$,
    \[ a \cdot y_i \notin R \cdot y_1 + R \cdot y_2 + \cdots + \widehat{R \cdot y_i} + \cdots + R \cdot y_n.\]
  \end{itemize}
\end{lemma}
\begin{proof}
  Let $K^*$ be an ``ultrapower'' of $K$, and let $R'$ be the subring
  of $K^*$ induced by $\tau_R$.  Note that $R'$ is a local ring of
  weight $n$, and $R' \supseteq R^*$.  Let $G_{n,m}$ be as in
  Lemma~\ref{scramble-strong}.  Let $\phi(x_1,\ldots,x_n)$ be the
  formula in $K^*$ expressing that for any $\vec{y} \in G_{n,m} \cdot
  \vec{x}$, we have
  \[ \forall i : \left(y_i \notin \sum_{j \ne i} R^* \cdot a^{-1} y_j\right).\]
  \begin{claim}
    Some tuple $\vec{z} \in (K^*)^n$ satisfies $\phi(\vec{x})$.
  \end{claim}
  \begin{claimproof}
    Because $R'$ has weight $n$, there is an $R'$-independent sequence
    $z_1,\ldots,z_n$.  We claim that $\phi(\vec{z})$ holds.  If
    $\vec{y} \in G_{n,m} \cdot \vec{z} \subseteq GL_n(K_0) \cdot
    \vec{z}$, then $\vec{y}$ is $R'$-independent by
    Lemma~\ref{spinaround}.  Therefore
    \[ y_i \notin \sum_{j \ne i} R' \cdot y_j.\]
    But $R^* \cup \{a^{-1}\} \subseteq R^* \cup K \subseteq R'$.  Therefore
    \[ y_i \notin \sum_{j \ne i} R^* \cdot a^{-1} \cdot y_j. \qedhere\]
  \end{claimproof}
  Now $\phi(x_1,\ldots,x_n)$ is a formula over $K$, so it is satisfied
  by some tuple $z_1,\ldots,z_n \in K^n$.  By choice of $G_{n,m}$ in
  Lemma~\ref{scramble-strong}, there is at least one $\vec{y} \in
  G_{n,m} \cdot \vec{z}$ such that $\vec{y}$ is scrambled.  By
  definition of $\phi$, the vector $\vec{y}$ has the desired
  properties.
\end{proof}

\subsection{Density in the Jacobson radical}
\begin{lemma}\label{jac-den}
  Let $R$ be a local integral domain of weight $n$, inducing a
  W-topology on $K = \Frac(R)$ that is also local of weight $n$.  Let
  $\widetilde{R}$ be the integral closure of $R$.  Then $R$ is dense
  in the Jacobson radical of $\widetilde{R}$, with respect to the
  topology induced by $\widetilde{R}$.
\end{lemma}
\begin{proof}
  Let $\tau$ and $\widetilde{\tau}$ denote the topologies induced by $R$
  and $\widetilde{R}$.  Then $\widetilde{\tau}$ is coarser than $\tau$.

  Write $\widetilde{R}$ as an intersection of pairwise incomparable
  valuation rings $\Oo_1 \cap \cdots \cap \Oo_m$.  The Jacobson
  radical is $\mm_1 \cap \cdots \cap \mm_n$, where $\mm_i$ is the
  maximal ideal of $\Oo_i$.  Let $x$ be an element of the Jacobson
  radical of $\widetilde{R}$.  For any $U \in \widetilde{\tau}$, we must show
  that $x + U$ intersects $R$.

  By taking $a$ small enough with respect to $\tau$, we can find non-zero $a$ such that
  \begin{itemize}
  \item $a \in R$
  \item $ax \in R$
  \item $a$ is in the Jacobson radical $\mm_1 \cap \cdots \cap \mm_n$.
  \item $a \widetilde{R} \subseteq U$.
  \end{itemize}
  Indeed, the first two requirements cut out $\tau$-neighborhoods of
  0, the third and fourth requirements cut out
  $\widetilde{\tau}$-neighborhoods of 0, and $\tau$ is finer than
  $\widetilde{\tau}$.

  By Lemma~\ref{re-slide}, there are $y_1, \ldots, y_n \in K$ such
  that
  \begin{itemize}
  \item For every $j \le m$, we have $\val_j(y_1) = \val_j(y_2) = \cdots =
    \val_j(y_n)$.
  \item For every $i \le n$, we have
    \[ a^2 y_i \notin \sum_{j \ne i} R y_j.\]
  \end{itemize}
  Scaling the $\vec{y}$'s, we may assume $y_1 = 1$.  Then $\val_j(y_i)
  = 0$ for all $i, j$.  In particular, $y_i \in \widetilde{R}$.

  Now $\wt(R) < n + 1$, so the set $\{1,x,ay_2,ay_3,\ldots,ay_n\}$
  cannot be $R$-independent.  One of three things happens:
  \begin{itemize}
  \item $1$ is in the $R$-module generated by $x, ay_2, ay_3, \ldots,
    ay_n$.  But this cannot happen, as the elements $x, ay_2, ay_3,
    \ldots, ay_n$ have positive valuation with respect to any of the
    $\val_i$'s, and $R \subseteq \Oo_i$.
  \item Say, $ay_2$ is in the $R$-module generated by $1, x, ay_3,
    ay_4, \ldots, ay_n$.  Then
    \begin{align*}
      ay_2 & \in R + Rx + \sum_{j = 3}^n R ay_j \\
      a^2 y_2 & \in Ra + Rax + \sum_{j = 3}^n R a^2 y_j.
    \end{align*}
    But $a, ax \in R$, and $a^2 \in R$, so that
    \[ a^2 y_2 \in Ra + Rax + \sum_{j = 3}^n R a^2 y_j. \subseteq R + R + \sum_{j = 3}^n R y_j = R y_1 + \sum_{j = 3}^n R y_j = \sum_{j \ne 2} R y_j,\]
    contradicting the choice of the $y$'s.
  \item $x$ is in the $R$-module generated by $1, ay_2, ay_3, \ldots,
    ay_n$.  Then there are $b, c_2, \ldots, c_n \in R$ such that
    \[ x = b + ac_2y_2 + ac_3y_3 + \cdots + ac_ny_n.\]
    But the $c_i \in R \subseteq \widetilde{R}$, and the $y_i \in \widetilde{R}$, so we see
    \[ b - x \in a\widetilde{R} \subseteq U.\]
    Then $b \in (x + U) \cap R$. \qedhere
  \end{itemize}
\end{proof}

\begin{lemma}\label{tops-not-same}
  Let $R$ be a local integral domain of weight $n$, inducing a
  W-topology on $K = \Frac(R)$ that is also local of weight $n$.  Let
  $\widetilde{R}$ be the integral closure of $R$.  Write
  $\widetilde{R}$ as $\Oo_1 \cap \cdots \cap \Oo_m$ for incomparable
  valuation rings $\Oo_i$.  Suppose the $\Oo_i$ are pairwise
  independent.  Then there is $a \in K^\times$ such that for every $b
  \in K^\times$,
  \[ R \cap b\Oo_2 \not\subseteq a\Oo_1.\]
\end{lemma}
\begin{proof}
  Let $\mm_i$ be the maximal ideal of $\Oo_i$.  Then $\bigcap_i \mm_i$
  is the Jacobson radical of $\widetilde{R}$.  Take non-zero $c \in
  \bigcap_i \mm_i$.  By the approximation theorem for V-topologies, we
  can find $a \in K$ such that
  \begin{align*}
    \val_1(a) &> \val_1(c) \\
    \val_i(a) &= \val_i(c), \text{ for } i > 1.
  \end{align*}
  Now, for every $b \in K^\times$, there is $u \in K^\times$ such that
  \begin{align*}
    \val_1(u) &= \val_1(c) > 0 \\
    \val_i(u) &> \max(\val_i(b),\val_i(c)) > 0, \text{ for } i > 1,
  \end{align*}
  by the approximation theorem for V-topologies.  Then $u \in
  \bigcap_i \mm_i$.  By Lemma~\ref{jac-den}, there are elements of $R$
  arbitrarily close to $u$ with respect to the
  $\widetilde{R}$-topology.  In particular, moving $u$, we can take $u
  \in R$.  Then $u \in R$ and $u \in b \Oo_2$, but $u \notin a \Oo_1$,
  since $\val_1(u) = \val_1(c) < \val_1(a)$.
\end{proof}

\section{Local components and V-topological coarsenings}

\subsection{The local case}
\begin{lemma}\label{tops2}
  Let $(K,\tau)$ be a local topological field of weight $n$.  Let
  $\tau_1, \tau_2$ be two distinct V-topological coarsenings of
  $\tau$.  Then for all $U \in \tau$ there is $V \in \tau_1$ such that
  for all $W \in \tau_2$,
  \[ U \cap W \not \subseteq V.\]
\end{lemma}
\begin{proof}
  The desired condition can be expressed as a local sentence in
  $(K,\tau,\tau_1,\tau_2)$.  Let $K^*$ be an ``ultrapower'' of $K$.
  Let $R, R_1, R_2$ be the rings induced by $\tau, \tau_1, \tau_2$,
  and let $\tau^*, \tau_1^*, \tau_2^*$ be the corresponding topologies
  on $K^*$.  Then $(K^*,\tau^*,\tau_1^*,\tau_2^*)$ is locally
  equivalent to $(K,\tau,\tau_1,\tau_2)$, by Proposition~\ref{p2}.  So
  we may work in $K^*$ instead.  Then $\tau^*$ is induced by $R$, and
  $\tau^*_i$ is induced by $R_i$.  Note that $R$ is a local ring of
  weight $n$, inducing a local topology of weight $n$.  By
  Proposition~\ref{v-coarses}, $R_i$ is one of the valuation rings
  whose intersection is the integral closure $\widetilde{R}$.  We must
  show
  \[ \forall c \ne 0 ~ \exists a \ne 0 ~ \forall b \ne 0 : cR \cap bR_2 \not \subseteq aR_1.\]
  Up to rescaling, we may assume $c = 1$, and then this is the content
  of Lemma~\ref{tops-not-same} (with $K$ and $K_0$ being $K^*$ and
  $K$, respectively).
\end{proof}

\begin{lemma}\label{incomparable}
  Let $(K,\tau)$ be a local topological field of weight $n$.  Let
  $K^*$ be an ``ultrapower'' of $K$.  Let $R$ be the subring of $K^*$
  induced by $\tau$.  Let $\widetilde{R}$ be the integral closure of $R$,
  and let $\Oo_1 \cap \cdots \cap \Oo_n$ be its decomposition into
  incomparable valuation rings.  Let $\mm_i$ be the maximal ideal of
  $\Oo_i$.  Then
  \[ \mm_i \cap R \not \subseteq \mm_j \cap R,\]
  for $i \ne j$.
\end{lemma}
\begin{proof}
  Let $\tau_i$ be the topology corresponding to $\Oo_i$.  Let $U \in
  \tau$ be a bounded neighborhood.  Let $i, j$ be given.  By
  Lemma~\ref{tops2}, there is $V \in \tau_j$ such that for all $W \in
  \tau_i$,
  \[ U \cap W \not \subseteq V.\]
  By saturation, there is $\epsilon \in U^* \cap \bigcap_{W \in
    \tau_i} W^*$ with $\epsilon \notin V^*$.  Then $\epsilon$ is a
  $\tau_i$-infinitesimal, so $\epsilon \in \mm_1$ by
  Lemma~\ref{esimals}.  Additionally, $\epsilon \in U^*$, so
  $\epsilon$ is $\tau$-bounded, and $\epsilon \in R$.  On the other
  hand, $\epsilon \notin V^*$, so $\epsilon$ is not a
  $\tau_j$-infinitesimal, and $\epsilon \notin \mm_j$.
\end{proof}

\begin{theorem}\label{local-thm}
  Let $(K,\tau)$ be a local W-topological field.
  \begin{enumerate}
  \item $\tau$ has a unique V-topological coarsening.
  \item If $K^*$ is an ``ultrapower,'' if $R \subseteq K^*$ is the
    ring induced by $\tau$, and $\widetilde{R}$ is its integral
    closure, then $\widetilde{R}$ is a valuation ring.
  \end{enumerate}
\end{theorem}
\begin{proof}
  The two statements are equivalent by Proposition~\ref{v-coarses}; we
  prove the second one.  Take the canonical decomposition $\widetilde{R} =
  \Oo_1 \cap \cdots \cap \Oo_n$.  Let $\mm_i$ denote the maximal ideal
  of $\mm$.  For each $i$, the intersection $\mm_i \cap R$ is a prime
  ideal $\pp_i \in \Spec R$.  By Lemma~\ref{incomparable}, the $\pp_i$
  are pairwise incomparable.

  Let $\pp$ be the maximal ideal of the local ring $R$.  By
  Chevalley's theorem, there is a valuation ring $\Oo$ with maximal
  ideal $\mm$, such that $\Oo \supseteq R$ and $\Oo \cap R = \mm$.
  Now $\Oo \supseteq R$ implies $\Oo \supseteq \widetilde{R}$, which
  in turn implies $\Oo \supseteq \Oo_i$ for some $i$ (\cite{prdf2},
  Corollary~6.8).  Then
  \begin{equation*}
    \Oo \supseteq \Oo_i \implies \mm \subseteq \mm_i \implies \mm \cap
    R \subseteq \mm_i \cap R.
  \end{equation*}
  Thus $\pp_i \supseteq \pp$.  As $\pp$ is the maximal ideal, we have
  $\pp_i = \pp$.  Then
  \[ \pp_i = \pp \supseteq \pp_j\]
  for all $j$, contradicting the pairwise incomparability of the
  $\pp_i$---unless $n = 1$.
\end{proof}

\subsection{Some commutative algebra}
Let $R$ be a domain.  We let $\MaxSpec R$ denote the set of maximal
ideals in $R$.
\begin{definition}
  A \emph{key localization} of $R$ is a
  localization $R_\pp$ for some maximal ideal $\pp \in \MaxSpec R$.
\end{definition}
We view $R_\pp$ as a subring of $K = \Frac(R)$.
\begin{proposition}\label{calg}
  Let $R$ be a domain.
  \begin{enumerate}
  \item \label{c1} $R$ equals the intersection of its key localizations.
  \item If $\pp_1, \pp_2$ are two distinct maximal ideals of $R$, then
    $R_{\pp_1}$ is incomparable to $R_{\pp_2}$.
  \item \label{c3} If $A \subseteq K$ is a local ring containing $R$, then $A$
    contains a key localization of $R$.
  \end{enumerate}
  Therefore, the key localizations of $R$ are the minimal local
  subrings of $K$ containing $R$, and they are in bijection with the
  maximal ideals of $R$.
\end{proposition}
\begin{proof}
  \begin{enumerate}
  \item Let $R'$ be the intersection of the key localizations.
    Clearly $R \subseteq R'$.  Conversely, suppose $x \notin R$.  Let
    $I = \{y \in R : xy \in R\}$.  Then $I$ is a proper ideal in $R$,
    because $1 \notin I$.  Take a maximal ideal $\pp$ containing $I$.
    We claim $x \notin R_\pp$.  Otherwise, $x = a/s$ for some $a \in
    R$ and $s \in R \setminus \pp$.  Then $s \in I$, contradicting $I
    \subseteq \pp$.  Thus $x \notin R_\pp$, and $x \notin R'$.
  \item Suppose $\pp_1, \pp_2$ are distinct.  Then $\pp_1, \pp_2$ are
    incomparable.  Take $x \in \pp_1 \setminus \pp_2$.  Then $1/x \in
    R_{\pp_2}$.  If $1/x \in R_{\pp_1}$, then $1/x = a/s$ for some $a
    \in R$ and $s \in R \setminus \pp_1$.  But then $s = ax \in
    \pp_1$, a contradiction.  So $1/x$ shows that $R_{\pp_2} \not
    \subseteq R_{\pp_1}$.  By symmetry, $R_{\pp_1} \not \subseteq
    R_{\pp_2}$.
  \item Let $\mm$ be the maximal ideal of $A$.  Then $\mm \cap R$ is a
    prime ideal in $R$, so $\mm \cap R \subseteq \pp$ for some $\pp
    \in \MaxSpec R$.  Then
    \begin{align*}
      x \in R & \implies x \in A \\
      x \in R \setminus \pp & \implies x \in A \setminus \mm \implies x^{-1} \in A.
    \end{align*}
    Therefore $R_\pp \subseteq A$. \qedhere
  \end{enumerate}
\end{proof}

\begin{lemma}\label{re-intersect}
  Let $R$ be a domain.  Let $R_1, \ldots, R_n$ be among the key
  localizations of $R$.  Let $R' = \bigcap_{i = 1}^n R_i$.  Then the
  key localizations of $R'$ are exactly $R_1, \ldots, R_n$.
\end{lemma}
\begin{proof}
  Let $K = \Frac(R)$.  By Proposition~\ref{calg}, it suffices to prove
  the following: if $A$ is a local subring of $K = \Frac(R)$, and $A
  \supseteq R'$, then $A \supseteq R_i$ for some $i$.

  Let $\mm$ be the maximal ideal of $A$.  Then $\mm \cap R$ is a prime
  ideal in $R$.  Let $\pp_i$ be the maximal ideal of $R$ whose
  localization is $R_i$.  By the prime avoidance lemma, one of two
  things happens:
  \begin{itemize}
  \item $\mm \cap R \subseteq \pp_i$ for some $i$.  Then $R_i
    \subseteq A$ as in the proof of Proposition~\ref{calg}(\ref{c3}).
  \item $\mm \cap R \not \subseteq \bigcup_{i = 1}^n \pp_i$.  In this
    case, take $x \in (\mm \cap R) \setminus \bigcup_{i = 1}^n \pp_i$.
    Then $1/x \in R_{\pp_i}$ for all $i$, so $1/x \in R' \subseteq A$.
    But this contradicts the fact that $x \in \mm = A \setminus
    A^\times$. \qedhere
  \end{itemize}
\end{proof}

\begin{lemma}\label{nice-local}
  Let $R$ be a local domain with maximal ideal $\pp$.  If the integral
  closure $\widetilde{R}$ is a local ring with maximal ideal $\qq$, then $\qq
  \cap R = \pp$.
\end{lemma}
\begin{proof}
  The intersection $\qq \cap R$ is a prime ideal in $R$, so $\qq \cap
  R \subseteq \pp$.  If equality does not hold, take $x \in \pp
  \setminus \qq$.  Then $x \notin \qq \implies 1/x \in \widetilde{R}$,
  so there exist $a_0, \ldots, a_{n-1} \in R$ such that
  \[ x^{-n} + a_{n-1} x^{-(n-1)} + \cdots + a_1x^{-1} + a_0 = 0,\]
  or equivalently,
  \[ -1 = a_{n-1} x + a_{n-2} x^2 + \cdots + a_1 x^{n-1} + a_0 x^n.\]
  The right side is in $\pp$ and the left is not, a contradiction.
\end{proof}

\begin{lemma}\label{nice-semilocal}
  Let $R$ be a domain with finitely many maximal ideals $\pp_1,
  \ldots, \pp_n$.  Suppose that for each key localization $R_{\pp_i}$,
  the integral closure $\widetilde{R_{\pp_i}}$ is a valuation ring.
  Then
  \begin{enumerate}
  \item \label{nr1} For each key localization $R_\pp$, the integral
    closure $\widetilde{R_\pp}$ is a key localization of the integral
    closure $\widetilde{R}$.
  \item \label{nr2} This map establishes a bijection from the key
    localizations of $R$ to the key localizations of $\widetilde{R}$.
  \end{enumerate}
\end{lemma}
\begin{proof}
  The integral closure $\widetilde{R}$ is the intersection of all
  valuation rings on $\Frac(R)$ containing $R$.  If $\Oo$ is a
  valuation ring, the following are equivalent:
  \begin{itemize}
  \item $\Oo \supseteq R$
  \item $\Oo \supseteq R_{\pp_i}$ for some $i$, by
    Proposition~\ref{calg}(\ref{c3}).
  \item $\Oo \supseteq \widetilde{R_{\pp_i}}$, since $\Oo$ is
    integrally closed.
  \end{itemize}
  Therefore, the $\widetilde{R_{\pp_i}}$ are the minimal valuation
  rings containing $R$, and
  \[ \widetilde{R} = \bigcap_{i = 1}^n \widetilde{R_{\pp_i}}.\]
  \begin{claim}
    The $\widetilde{R_{\pp_i}}$ are pairwise incomparable.
  \end{claim}
  \begin{claimproof}
    Let $\mm_i$ be the maximal ideal of the valuation ring
    $\widetilde{R_{\pp_i}}$.  By Lemma~\ref{nice-local}, $\mm_i$
    restricts to the maximal ideal on $R_{\pp_i}$, and then to the
    ideal $\pp_i$ on $R$.  That is, $\mm_i \cap R = \pp_i$.  Then for
    any $i, j$
    \begin{equation*}
      \widetilde{R_{\pp_i}} \subseteq \widetilde{R_{\pp_j}} \iff \mm_i
      \supseteq \mm_j \implies \mm_i \cap R \supseteq \mm_j \cap R
      \iff \pp_i \supseteq \pp_j.
    \end{equation*}
    The $\pp_i$ are pairwise incomparable, being maximal ideals,
    and so the claim is proven.
  \end{claimproof}
  Therefore, the $\widetilde{R_{\pp_i}}$ are pairwise incomparable
  valuation rings, with intersection $\widetilde{R}$.  By Proposition
  6.2(7) in \cite{prdf2}, the $\widetilde{R_{\pp_i}}$ are the key
  localizations of $\widetilde{R}$.  The map $R_\pp \mapsto
  \widetilde{R_\pp}$ is injective by the Claim.
\end{proof}

\subsection{The non-local case}

\begin{theorem}\label{thm:non-loc1}
  Let $\tau$ be a W-topology on $K$.  Let $\tau_1, \ldots, \tau_n$ be
  the local components of $\tau$, in the sense of
  Definition~\ref{def:comps}.
  \begin{itemize}
  \item Each $\tau_i$ has a unique V-topological coarsening.
  \item This establishes a bijection between the local components of
    $\tau$ and the V-topological coarsenings of $\tau$.
  \end{itemize}
\end{theorem}
\begin{proof}
  The first point follows by Theorem~\ref{local-thm}.  For the second
  point, take an ``ultrapower'' $K^*$.  Let $R$ and $R_i$ be the
  $\vee$-definable rings induced by $\tau$ and $\tau_i$.  By
  Definition~\ref{def:comps}, the $R_i$ are the key localizations of
  $R$.  By Theorem~\ref{local-thm}, each integral closure
  $\widetilde{R_i}$ is a valuation ring, corresponding to the unique
  V-topological coarsening of $\tau_i$.  By
  Lemma~\ref{nice-semilocal}, the $\widetilde{R_i}$ are pairwise
  distinct, and are exactly the key localizations of $\widetilde{R}$.
  By Proposition~\ref{v-coarses}, the key localizations of
  $\widetilde{R}$ correspond exactly to the the V-topological
  coarsenings of $\tau$.
\end{proof}

\begin{definition}
  Let $\tau, \tau_1, \ldots, \tau_n$ be topologies on $K$.  Then
  $\tau$ is an \emph{independent sum} of $\tau_1, \ldots, \tau_n$ if
  the diagonal map
  \[ (K,\tau) \to (K,\tau_1) \times \cdots \times (K,\tau_n)\]
  is a homeomorphism onto its image, and the image is dense.
\end{definition}
More explicitly, this means that the following conditions hold:
\begin{itemize}
\item The following is a filter basis for $\tau$:
  \begin{equation*}
    \{U_1 \cap \cdots \cap U_n : U_1 \in \tau_1, ~ U_2 \in \tau_2,
    \ldots, U_n \in \tau_n\}.
  \end{equation*}
\item If $a_i \in K$ and $U_i \in \tau_i$ for all $i$, then
  $\bigcap_{i = 1}^n (a_i + U_i) \ne \emptyset$.
\end{itemize}
In other words, the $\tau_i$ are jointly independent, and they
generate $\tau$.
\begin{lemma}\label{associator}
  If $\sigma$ is an independent sum of $\tau_1, \ldots, \tau_{n-1}$,
  and $\tau$ is an independent sum of $\sigma$ and $\tau_n$, then
  $\tau$ is an independent sum of $\tau_1, \ldots, \tau_n$.
\end{lemma}
\begin{proof}
  Let $\mathcal{F}$ be the class of topological embeddings with dense
  image.  Then $\mathcal{F}$ is closed under composition.
  Additionally, if $f : X \to Y$ is in $\mathcal{F}$ and $Z$ is
  another topological space, then $X \times Z \to Y \times Z$ is in
  $\mathcal{F}$.  By assumption, the diagonal maps
  \begin{align*}
    (K,\sigma) &\to (K,\tau_1) \times \cdots \times (K,\tau_{n-1}) \\
    (K,\tau) & \to (K,\sigma) \times (K,\tau_{n-1})
  \end{align*}
  are in $\mathcal{F}$.  Therefore $\mathcal{F}$ also contains the
  composition
  \begin{equation*}
    (K,\tau) \to (K,\sigma) \times (K,\tau_{n}) \to (K,\tau_1)
    \times \cdots \times (K,\tau_{n-1}) \times (K,\tau_n),
  \end{equation*}
  which is the diagonal map.
\end{proof}

\begin{theorem}\label{thm:non-loc2}
  Let $\tau$ be a W-topology on a field $K$.  Let $\tau_1, \ldots,
  \tau_n$ be the local components of $\tau$.  Then $\tau$ is an
  independent sum of the $\tau_i$.
\end{theorem}
\begin{proof}
  Fix an ``ultrapower'' $K^*$ of $K$.  Let $R, R_1, \ldots, R_n$ be
  the $\vee$-definable rings corresponding to $\tau, \tau_1, \ldots,
  \tau_n$.  The $R_i$ are the key localizations of $R$.  For $i \le
  n$, let $S_i = R_1 \cap \cdots \cap R_i$.  Each $S_i$ is a
  $K$-algebra that is $\vee$-definable over $K$.  Moreover, $S_i
  \supseteq R$, so each $S_i$ is a $W_n$-ring (Lemma~2.7 in
  \cite{prdf5}).  By Proposition~\ref{wdict-2}, each $S_i$ corresponds
  to a W-topology $\sigma_i$.  Additionally,
  \begin{align*}
    S_1 &= R_1 \\
    S_n &= R_1 \cap \cdots \cap R_n = R,
  \end{align*}
  by Proposition~\ref{calg}(\ref{c1}).  Thus $\sigma_1 = \tau_1$, and
  $\sigma_n = \tau$.  By Lemma~\ref{re-intersect}, the key
  localizations of $S_i$ are $R_1, \ldots, R_i$, and so the local
  components of $\sigma_i$ are $\tau_1, \ldots, \tau_i$.
  \begin{claim}\label{exclaim}
    $\sigma_i$ is an independent sum of $\sigma_{i-1}$ and $\tau_i$.
  \end{claim}
  \begin{claimproof}
    We first prove independence.  Let $\widetilde{\tau_j}$ denote the
    unique V-topological coarsening of $\tau_j$.  By
    Theorem~\ref{thm:non-loc1},
    \begin{itemize}
    \item the V-topological coarsenings of $\sigma_{i-1}$ are
      $\{\widetilde{\tau_1}, \ldots, \widetilde{\tau_{i-1}}\}$,
    \item the V-topological coarsenings of $\tau_i$ are
      $\{\widetilde{\tau_i}\}$,
    \end{itemize}
    and these two sets do not overlap.  So $\sigma_{i-1}$ and $\tau_i$
    have no common V-topological coarsenings.  As $\tau_i$ and
    $\sigma_{i-1}$ are both coarsenings of the original W-topology
    $\tau$, Theorem~7.16 of \cite{prdf5} applies, and $\tau_i$ and
    $\sigma_{i-1}$ are independent.

    It remains to show that $\sigma_{i-1}$ and $\tau_i$ generate
    $\sigma_i$.  Consider the topologies $\sigma_{i-1}^*, \tau_i^*,
    \sigma_i^*$ on $K^*$ induced by $S_{i-1}, R_i,$ and $S_i$,
    respectively.  Note $S_i = S_{i-1} \cap R_i$.  For any non-zero
    $a$, there are non-zero $b, c$ such that
    \[  b S_{i-1} \cap c R_i \subseteq aS_i.\]
    Indeed, if we take $b = c = a$, then equality holds.  This proves
    the local sentence
    \begin{equation}
      \forall U \in \sigma_i^* ~ \exists V \in \sigma_{i-1}^* ~
      \exists W \in \tau_i^* : V \cap W \subseteq U. \label{coinit1}
    \end{equation}
    Conversely,
    \begin{equation}
      \forall V \in \sigma_{i-1}^* ~ \forall W \in \tau_i^* ~ \exists
      U \in \sigma_i^* : U \subseteq V \cap W, \label{coinit2}
    \end{equation}
    because we can take $U = V \cap W$.  (Both $V$ and $W$ are in
    $\sigma_i^*$, because $\sigma_i^*$ is finer than $\sigma_{i-1}^*$
    and $\tau_i^*$.)  Equations (\ref{coinit1}) and (\ref{coinit2})
    are local sentences, so they hold for $\sigma_{i-1}, \tau_i,$ and
    $\sigma_i$, by Proposition~\ref{p2}(\ref{p2e3}).  That is,
    \begin{align*}
      \forall U \in \sigma_i ~ \exists V \in \sigma_{i-1} ~
      \exists W \in \tau_i &: V \cap W \subseteq U \\
      \forall V \in \sigma_{i-1} ~ \forall W \in \tau_i ~ \exists
      U \in \sigma_i &: U \subseteq V \cap W.
    \end{align*}
    This expresses that $\sigma_i$ is generated by $\sigma_{i-1}$ and
    $\tau_i$, finishing the proof of the Claim.
  \end{claimproof}
  Combining the Claim with Lemma~\ref{associator}, we see by induction
  on $i$ that $\sigma_i$ is an independent sum of $\tau_1, \ldots,
  \tau_i$.  Taking $i = n$, we get the desired result.
\end{proof}

\begin{corollary}\label{sq-weird}
  Let $(K,\tau)$ be a W-topological field.
  \begin{enumerate}
  \item\label{sw1} If $\characteristic(K) \ne 2$ and the squaring map
    $K^\times \to K^\times$ is an open map, then $\tau$ is local
    and has a unique V-topological coarsening.
  \item If $\characteristic(K) = p > 0$ and the Artin-Schreier map $K
    \to K$ is an open map, then $\tau$ is local and has a unique
    V-topological coarsening.
  \end{enumerate}
\end{corollary}
\begin{proof}
  Let $\tau_1, \ldots, \tau_n$ be the local components of $\tau$.  By
  Theorem~\ref{thm:non-loc2}, $\tau$ is an independent sum of $\tau_1,
  \ldots, \tau_n$.  By Theorem~\ref{thm:non-loc1}, $n$ is the number
  of V-topological coarsenings.  Suppose for a contradiction that $n >
  1$.  If $\characteristic(K) \ne 2$, then the squaring map is not
  open, by the proof of Claim~6.9 in \cite{prdf5}.  Essentially, one
  chooses $x$ to be infinitesimally close to 1 with respect to
  $\tau_1, \tau_3, \tau_5, \ldots$, and infinitesimally close to $-1$
  with respect to $\tau_2, \tau_4, \tau_6, \ldots$.  Then $x \not
  \approx 1$ and $x^2 \approx 1$ with respect to $\tau$, which
  contradicts squaring being an open map.  The Artin-Schreier case is
  similar, using $0, 1$ instead of $1, -1$.
\end{proof}

\begin{corollary}
  ~
  \begin{enumerate}
  \item\label{end1.5} If $(K,+,\cdot,\ldots)$ is an unstable dp-finite
    field, then the canonical topology on $K$ is a local W-topology.
  \item\label{end1}  If $(K,+,\cdot,\ldots)$ is an unstable dp-finite field, then
    $K$ admits a unique definable V-topology.
  \item\label{end2} If $(K,+,\cdot,v)$ is a dp-finite valued field,
    then $v$ is henselian.
  \item\label{end3} If $(K,+,\cdot)$ is a dp-finite field that is
    neither finite, nor algebraically closed, nor real closed, then
    $K$ admits a non-trivial definable henselian valuation.
  \item\label{end4} The conjectural classification of dp-finite fields
    holds, as in Theorem~3.11 of \cite{halevi-hasson-jahnke}.
  \end{enumerate}
\end{corollary}
\begin{proof}
  The canonical topology on $K$ is a W-topology by Theorem~6.3 in
  \cite{prdf5}.  By Proposition~5.17(4-5) in \cite{prdf2} and
  compactness, the canonical topology has the following properties:
  \begin{itemize}
  \item For every neighborhood $U \ni 1$ there is a neighborhood $V
    \ni 1$ such that
    \begin{equation*}
      V \subseteq \{x^2 : x \in U\}.
    \end{equation*}
    Equivalently, the squaring map $K^\times \to K^\times$ is an open
    map.
  \item If $\characteristic(K) = p$, then for every neighborhood $U
    \ni 0$ there is a neighborhood $V \ni 0$
    \begin{equation*}
      V \subseteq \{x^p - x : x \in U\}.
    \end{equation*}
    Equivalently, the Artin-Schreier map $K \to K$ is an open map.
  \end{itemize}
  By Corollary~\ref{sq-weird}, the canonical topology is local and has
  a unique V-topological coarsening.  This proves part \ref{end1.5}.
  By Theorem~6.6 in \cite{prdf5}, the V-topological coarsenings of the
  canonical topology are exactly the definable V-topologies.  This
  proves part \ref{end1}.  The remaining points then follow by (the
  proof of) Proposition~6.4 in \cite{prdf4}.
\end{proof}
We note another characterization of local W-topologies.
\begin{proposition}\label{lchar}
  Let $\tau$ be a W-topology on a field $K$.  The following are equivalent:
  \begin{enumerate}
  \item \label{lchar1} $\tau$ is not local.
  \item \label{lchar2} $\tau$ is an independent sum of two
    W-topologies.
  \item \label{lchar3} $\tau$ is an independent sum of finitely many
    W-topologies.
  \end{enumerate}
\end{proposition}
\begin{proof}
  The implication (\ref{lchar1})$\implies$(\ref{lchar2}) follows by
  the Theorem~\ref{thm:non-loc2}, or rather by Claim~\ref{exclaim} in
  the proof.  In the notation of the proof, $\tau$ equals $\sigma_n$,
  which is an independent sum of $\sigma_{n-1}$ and $\tau_n$.

  The implication (\ref{lchar2})$\implies$(\ref{lchar3}) is trivial.
  It remains to prove (\ref{lchar3})$\implies$(\ref{lchar1}).  Suppose
  $\tau$ is an independent sum of $\tau_1, \ldots, \tau_n$, but $\tau$
  is local.
  \begin{claim}
    If $B_i$ is $\tau_i$-bounded for $i = 1, \ldots, n$, then
    $\bigcap_{i = 1}^n B_i$ is $\tau$-bounded.
  \end{claim}
  \begin{claimproof}
    By Lemma~2.1(e) in \cite{PZ}, a set $B$ is bounded if for every
    neighborhood $U$, there is a neighborhood $V$ such that $B \cdot V
    \subseteq U$.  Let $U_1 \cap \cdots \cap U_n$ be a basic
    neighborhood in $\tau$, so that each $U_i$ is a neighborhood in
    $\tau_i$.  Then there are $V_i \in \tau_i$ such that $B_i \cdot
    V_i \subseteq U_i$.  Then
    \begin{equation*}
      \left( \bigcap_{i = 1}^n B_i \right) \cdot \left( \bigcap_{i =
        1}^n V_i \right) \subseteq \bigcap_{i = 1}^n U_i,
    \end{equation*}
    and $\bigcap_{i = 1}^n V_i$ is in $\tau$.
  \end{claimproof}
  For each $i$, take a $\tau_i$-bounded set $B_i$ containing both $0$
  and $1$ in its interior (i.e., $B_i \in \tau_i$ and $B_i - 1 \in
  \tau_i$).  Let $B = B_1 \cap \cdots \cap B_n$.  Then $B$ is
  $\tau$-bounded, by the Claim.  Because $\tau$ is local, there is a
  $\tau$-bounded set $C$ such that
  \[ \forall x \in B : (1/x \in C \text{ or } 1/(1-x) \in C).\]
  The fact that $\tau$ is finer than $\tau_i$ implies that $\tau$ has
  fewer bounded sets, so $C$ is $\tau_i$-bounded for each $i$.

  For each $i$, take small enough $D_i \in \tau_i$ to ensure that
  \begin{itemize}
  \item $D_i \cdot C$ is contained in the neighborhood $K \setminus
    \{1\}$, i.e., $1 \notin D_i \cdot C$.
  \item $D_i \subseteq B_i$ and $1 - D_i \subseteq B_i$.
  \end{itemize}
  By independence, there is $x \in (1 - D_1) \cap D_2 \cap D_3 \cap
  \cdots \cap D_n$.  Then $x \in B_1 \cap \cdots \cap B_n = B$, so one
  of two things happens:
  \begin{itemize}
  \item $1/x \in C$.  But $x \in D_2$, so then $1 = x \cdot (1/x) \in
    D_2 \cdot C$, a contradiction.
  \item $1/(1-x) \in C$.  But $1 - x \in D_1$, so then $1 = (1-x)
    \cdot (1/(1-x)) \in D_1 \cdot C$, a contradiction. \qedhere
  \end{itemize}
\end{proof}

\section{Remaining questions}
We leave the following natural questions to future work.
\begin{enumerate}
\item By Theorem~\ref{thm:non-loc2} and Proposition~\ref{lchar}, every
  W-topology decomposes into an independent sum of indecomposable
  W-topologies.  Is this decomposition unique?
\item If a W-topology $\tau$ decomposes as an independent sum of
  $\tau_1, \ldots, \tau_n$, is it true that $\wt(\tau) = \wt(\tau_1) +
  \cdots + \wt(\tau_n)$?
\item If $\tau_1, \ldots, \tau_n$ are jointly independent
  W-topologies, do they generate a W-topology?
\item If $\tau_1, \tau_2$ are two W-topologies without any common
  V-topological coarsenings, then are $\tau_1, \tau_2$ necessarily
  independent?
\item Local topologies of weight 2 on fields of characteristic 0 are
  exactly the ``DV-topologies'' of (\cite{prdf4}, Definition~8.18).
  This follows by the classification in \S 8.2 of \cite{prdf5}.  Can
  this sort of classification be generalized to higher ranks, or
  positive characteristic?
\item Let $(K,\tau)$ be a W-topological field, and $K^*$ be an
  ``ultrapower.''  Let $R$ be the associated $\vee$-definable ring of
  $K$-bounded elements in $K^*$.  Let $I$ be the type-definable ideal
  of $K$-infinitesimal elements in $K^*$.  Is $R/I$ always an artinian
  ring of length equal to $\wt(\tau)$?
\item In several places, we used the special properties of the rings
  $R = R_\tau$ arising in Proposition~\ref{p1}.  For example,
  \begin{itemize}
  \item If $\wt(R) = n$, then $R$ induces a topology of weight exactly
    $n$.
  \item $R$ is local if and only if $R$ induces a local W-topology.
  \end{itemize}
  Is there a natural algebraic condition implying these properties,
  and satisfied by the rings $R_\tau$ of Proposition~\ref{p1}?  If so,
  there may be a more natural class of rings hidden inside the
  $W_n$-rings.  This natural class would probably include
  multi-valuation rings $\Oo_1 \cap \cdots \cap \Oo_n$ for which the
  $\Oo_i$ are pairwise \emph{independent}, as well as the ring $R$ of
  \S 8.4 in \cite{prdf4}.
\item Do the global fields $\Qq$ and $\Ff_p(t)$ support any local
  W-topologies other than the usual V-topologies?
\item If $K$ is an unstable dp-finite field, does every heavy
  definable set have non-empty interior?  How much ``tame topology''
  can we prove?
\item Do the techniques used to analyze dp-finite fields have any
  generalizations to finite-burden fields, or strongly dependent
  fields?
\item Let $(K,+,\cdot,\ldots)$ be an NIP field, possibly with extra
  structure.  Must every definable Hausdorff non-discrete field
  topology on $K$ be a W-topology?
\end{enumerate}

\begin{acknowledgment}
The author would like to thank Meng Chen for hosting the author at
Fudan University, where this research was carried out.  {\tiny This
  material is based upon work supported by the National Science
  Foundation under Award No. DMS-1803120.  Any opinions, findings, and
  conclusions or recommendations expressed in this material are those
  of the author and do not necessarily reflect the views of the
  National Science Foundation.}
\end{acknowledgment}

\bibliographystyle{plain} \bibliography{mybib}{}

\end{document}